\documentclass[draft]{amsart}
\usepackage{amsmath, amssymb}

\newtheorem{thm}{Theorem}
\newtheorem{lem}[thm]{Lemma}
\newtheorem{cor}[thm]{Corollary}

\theoremstyle{remark}
\newtheorem{rmk}[thm]{Remark}

\numberwithin{thm}{section}
\numberwithin{equation}{section}

\newcommand{\cv}{\mathbf{C}}

\newcommand{\rv}{\mathbf{R}}
\newcommand{\nv}{\mathbf{N}}
\newcommand{\aut}{\textup{Aut}}
\newcommand{\Jac}{\textup{Jac}}
\newcommand{\grad}{\textup{grad}}
\newcommand{\diag}{\textup{diag}}
\newcommand{\Diag}{\textup{Diag}}

\begin{document}

\title[The degree of automorphisms of quasi-circular domains]{On the degree of automorphisms of quasi-circular domains fixing the origin}

\author{Feng Rong}

\address{Department of Mathematics, School of Mathematical Sciences, Shanghai Jiao Tong University, 800 Dong Chuan Road, Shanghai, 200240, P.R. China}
\email{frong@sjtu.edu.cn}


\thanks{The author is partially supported by the National Natural Science Foundation of China (Grant No. 11371246).}

\begin{abstract}
By using the Bergman representative coordinates, we give the necessary and sufficient condition for the degree of automorphisms of quasi-circular domains fixing the origin to be equal to the resonance order, thus solving a conjecture of the author.
\end{abstract}

\maketitle

\section{Introduction}

Denote by $\rho$ the linear circle action on $\cv^n$ as follows:
$$\rho:\cv^n\rightarrow \cv^n;\ z=(z_1,\cdots,z_n)\mapsto \rho(z)=(e^{im_1\theta}z_1,\cdots,e^{im_n\theta}z_n),$$
where $m_i$, $1\le i\le n$, are positive integers, and $\theta\in \rv$. A bounded domain $D$ of $\cv^n$ is called a \textit{quasi-circular} domain of weight $m=(m_1,\cdots,m_n)$, if $\rho(D)=D$. Without loss of generality, we will assume that $1\le m_1\le \cdots\le m_n$ and $\textup{gcd}(m_1,\cdots,m_n)=1$. And we always assume that $0\in D$.

In \cite{K:auto}, Kaup showed that all automorphisms of quasi-circular domains fixing the origin are polynomial mappings. And we gave a uniform upper bound for such mappings in \cite{R:quasi}, in terms of the so-called ``quasi-resonance order". Also in \cite{R:quasi}, we conjectured that the optimal upper bound should be given by the so-called ``resonance order". In a recent work \cite{YZ:quasi}, the authors showed that, in dimension two, this conjecture does not hold in general and indeed the upper bound given by the quasi-resonance order is optimal (cf. \cite[Example 5.1]{DR:auto}).

The main purpose of this paper is to give the necessary and sufficient condition for the above mentioned conjecture to hold, i.e. when the optimal upper bound is given by the resonance order.

Assume that for $0=k_0<\cdots<k_l=n$ one has 
\begin{equation}\label{E:m1}
m_{k_p+1}=\cdots=m_{k_{p+1}},\ \ \ 0\le p\le l-1
\end{equation}
and
\begin{equation}\label{E:m2}
m_{k_p}<m_{k_p+1},\ \ \ 1\le p\le l-1.
\end{equation}
Then, our main result is the following

\begin{thm}\label{T:main}
Let $D$ be a bounded quasi-circular domain containing the origin, and of weight $m=(m_1,\cdots,m_n)$, with $m$ satisfying \eqref{E:m1} and \eqref{E:m2}. Then all automorphisms of $D$ fixing the origin are polynomial mappings with degree less than or equal to the resonance order if and only if the linear part of each of such automorphisms is of the form $\Diag(A_1,\cdots,A_l)\cdot z^t$, where $A_p$ is a $(k_p-k_{p-1})\times (k_p-k_{p-1})$ matrix, $1\le p\le l$.
\end{thm}

In section \ref{S:pre}, we recall the definitions of the resonance order and the Bergman representative coordinates. In section \ref{S:degree}, we prove Theorem \ref{T:main}.

\section{Preliminaries}\label{S:pre}

Let $\nv$ be the set of nonnegative integers and $\alpha=(\alpha_1,\cdots,\alpha_n)\in \nv^n$. Denote $|\alpha|=\alpha_1+\cdots+\alpha_n$ and $m\cdot \alpha=m_1\alpha_1+\cdots+m_n\alpha_n$.

For $1\le i\le n$, define the \textit{i-th resonance set} as
$$E_i:=\{\alpha:\ m\cdot \alpha=m_i\},$$
and the \textit{i-th resonance order} as
$$\mu_i:=\max\{|\alpha|:\ \alpha\in E_i\}.$$
Note that $\mu_i\le m_i$. Then, define the \textit{resonance set} as
$$E:=\bigcup\limits_{i=1}^n E_i,$$
and the \textit{resonance order} as
$$\mu:=\max\{|\alpha|:\ \alpha\in E\}=\max\limits_{1\le i\le n} \mu_i.$$

A polynomial $P$ is said to be \textit{m-homogeneous of order k} if, for any $\lambda\in \cv$, $P(\lambda^{m_1}z_1,\cdots,\lambda^{m_n}z_n)=\lambda^k P(z_1,\cdots,z_n)$. When $k=m_i$ for some $1\le i\le n$, one says that $P$ is an \textit{i-th resonant polynomial}.

Let $K_D(z,w)$ be the Bergman kernel, i.e. the reproducing kernel of the space of square integrable holomorphic functions on $D$. Since $D$ is bounded, one has $K_D(z,z)>0$ for $z\in D$. The Bergman metric tensor $T_D(z,w)$ is defined as the $n\times n$ matrix with entries $t_{ij}^D(z,w)=\frac{\partial^2}{\partial \bar{w}_i\partial z_j}\log K_D(z,w)$, $1\le i,j\le n$. For $z\in D$, one knows that $T_D(z,z)$ is a positive definite Hermitian matrix (see e.g. \cite{B:Book}), and one has the following transformation formula for the Bergman metric tensor,
\begin{equation}\label{E:TD}
T_D(z,w)=\overline{\Jac_f^t(w)}T_D(f(z),f(w))\Jac_f(z),\ \ \ f\in \aut(D).
\end{equation}

The Bergman representative coordinates at $\xi\in D$ is defined as (see, e.g. \cite{GKK:Book, Lu})
\begin{equation}\label{E:B}
\sigma_\xi^D(z):=T_D(\xi,\xi)^{-1} \left.\grad_{\bar{w}}\log \frac{K_D(z,w)}{K_D(w,w)}\right|_{w=\xi}.
\end{equation}
As in \cite{LR:quasi}, one readily checks that $K_D(z,0)\equiv K_D(0,0)$ when $D$ is a quasi-circular domain. Thus the Bergman representative coordinates $\sigma_0^D(z)$ is defined for all $z\in D$. Note that as in \cite{YZ:quasi}, the power of $T_D(\xi,\xi)$ is set to be $-1$ instead of $-1/2$.

\section{The degree of automorphisms}\label{S:degree}

Throughout this section, we assume that $D$ is a quasi-circular domain containing the origin, and $f$ is an automorphism of $D$ with $f(0)=0$. Write $\sigma(z)$ for $\sigma_0^D(z)$, $J_f(z)$ for the linear part of $f(z)$, i.e. $J_f(z)=\Jac_f(0)\cdot z^t$, and $\zeta=(\zeta_1,\cdots,\zeta_n)=\sigma(z)=(\sigma_1(z),\cdots,\sigma_n(z))$.

\begin{lem}\label{L:resonant}
For each $1\le i\le n$, $\sigma_i(z)$ is of the form $\sigma_i(z)=z_i+g_i(z)$, where $g_i(z)$ contains only nonlinear $i$-th resonant monomials.
\end{lem}
\begin{proof}
Applying \eqref{E:TD} to $\rho$, one gets
\begin{equation}\label{E:TD_rho}
T_D(z,w)=\overline{\Jac_\rho^t(w)}T_D(\rho(z),\rho(w))\Jac_\rho(z).
\end{equation}
Setting $w=0$ in \eqref{E:TD_rho}, one obtains
\begin{equation}\label{E:TD0}
T_D(z,0)=\diag(e^{-im_1\theta},\cdots,e^{-im_n\theta})T_D(\rho(z),0)\diag(e^{im_1\theta},\cdots,e^{im_n\theta}).
\end{equation}

Write $t_{ij}^D(z,0)=\sum\limits_{|\alpha|\ge 0} a_{ij}^\alpha z^\alpha=\sum\limits_{|\alpha|\ge 0} a_{ij}^\alpha z_1^{\alpha_1}\cdots z_n^{\alpha_n}$, $1\le i,j\le n$. Then from \eqref{E:TD0}, one gets
\begin{equation}\label{E:a}
\sum\limits_{|\alpha|\ge 0} a_{ij}^\alpha z^\alpha=e^{i(m_j-m_i)\theta}\sum\limits_{|\alpha|\ge 0} a_{ij}^\alpha e^{i(m\cdot \alpha)\theta}z^\alpha.
\end{equation}
Since \eqref{E:a} holds for any $\theta\in \rv$, $a_{ij}^\alpha$ can be nonzero only when
$$(m_j-m_i)+m\cdot \alpha=0,$$
which is satisfied if and only if
\begin{equation}\label{E:a1}
\alpha+e_j\in E_i.
\end{equation}
Here $e_j$ denotes the multi-index with the $j$-th entry equal to 1 and all other entries zero.

Therefore, one can write $T_D(z,0)$ as
\begin{equation}\label{E:M}
T_D(z,0)=T_D(0,0)+M(z).
\end{equation}
Here $T_D(0,0)=\Diag(T_1,\cdots,T_l)$, where $T_p$, $1\le p\le l$, is a $(k_p-k_{p-1})\times (k_p-k_{p-1})$ matrix. And $M(z)=[M_{pq}(z)]_{1\le p,q\le l}$, where $M_{pq}(z)$ is a $(k_p-k_{p-1})\times (k_q-k_{q-1})$ matrix and $M_{pq}=0$ for $p\le q$.

By \eqref{E:B} and \eqref{E:M}, one has
\begin{equation}\label{E:Jac}
\Jac_\sigma(z)=T_D(0,0)^{-1}T_D(z,0)=I_n+T_D(0,0)^{-1}M(z)=:I_n+N(z).
\end{equation}
From the form of $T_D(0,0)$ and $M(z)$, one gets that $N(z)=[N_{pq}(z)]_{1\le p,q\le l}$, where $N_{pq}(z)=0$ for $p\le q$ and $N_{pq}(z)=T_p^{-1}M_{pq}(z)$ for $p>q$.

Since $\sigma(0)=0$, from \eqref{E:Jac} and \eqref{E:a1}, one sees that $\sigma_i(z)$, $1\le i\le n$, is of the desired form.
\end{proof}

\begin{cor}\label{C:inverse}
The Bergman mapping $\sigma(z)$ is invertible. Moreover, writing $\sigma^{-1}(\zeta)\\=(\tau_1(\zeta),\cdots,\tau_n(\zeta))$, for each $1\le i\le n$, $\tau_i(\zeta)$ is of the form $\tau_i(\zeta)=\zeta_i+h_i(\zeta)$, where $h_i(\zeta)$ contains only nonlinear $i$-th resonant monomials.
\end{cor}
\begin{proof}
By Lemma \ref{L:resonant}, it is clear that $\sigma(z)$ is invertible.

For $1\le i\le k_1$, one has $\zeta_i=\sigma_i\circ\sigma^{-1}(\zeta)=\tau_i(\zeta)$. Thus, $h_i(\zeta)=0$.

For $k_1+1\le i\le k_2$, one has $\zeta_i=\sigma_i\circ\sigma^{-1}(\zeta)=\tau_i(\zeta)+g_i\circ\sigma^{-1}(z)$. Thus,
$$h_{k_1+1}(\zeta)=-g_{k_1+1}\circ\sigma^{-1}(z)=-g_{k_1+1}(\zeta_1,\cdots,\zeta_{k_1},0,\cdots,0),$$
which contains only nonlinear $(k_1+1)$-th resonant monomials.
$$\begin{aligned}
h_{k_1+2}(\zeta)&=-g_{k_1+2}\circ\sigma^{-1}(z)\\
&=-g_{k_1+2}(\zeta_1,\cdots,\zeta_{k_1},\zeta_{k_1+1}-g_{k_1+1}(\zeta_1,\cdots,\zeta_{k_1},0,\cdots,0),0,\cdots,0),
\end{aligned}$$
which contains only nonlinear $(k_1+2)$-th resonant monomials. By induction, one gets for every $1\le i\le n$,
$$h_i(\zeta)=-g_i(\tau_1(\zeta),\cdots,\tau_{i-1}(\zeta),0,\cdots,0),$$
which contains only nonlinear $i$-th resonant monomials.
\end{proof}

\begin{lem}\label{L:linear}
$f(z)=\sigma^{-1}\circ J_f\circ \sigma(z)$.
\end{lem}
\begin{proof}
It is well-known that in Bergman representative coordinates, there exists a linear map $L$ such that $\sigma\circ f=L\circ \sigma$. Since $\Jac_\sigma(0)$ is the identity, one has $L=J_f$. Therefore, the lemma follows from Corollary \ref{C:inverse}.
\end{proof}

It is easy to see that Theorem \ref{T:main} follows from the above lemma. Moreover, one has the following

\begin{cor}
Let $D$ be a bounded quasi-circular domain containing the origin and $f=(f_1,\cdots,f_n)$ an automorphism of $D$ fixing the origin. Then the degree of $f$ is less than or equal to the resonance order if and only if each $f_i$, $1\le i\le n$, is an $i$-th resonant polynomial.
\end{cor} 

\begin{rmk}
Lemma \ref{L:linear} gives a complete description of all possible forms of automorphisms of quasi-circular domains fixing the origin. It also gives an alternative definition of the ``quasi-resonance order", which is easier to use and compute.
\end{rmk}

\end{document}